\documentclass[11pt]{amsart}

\usepackage{amssymb}
\usepackage{tikz-cd}

\usepackage[
  left=1.2in,
  right=1.2in,
  top=1.2in,
  bottom=1.2in
]{geometry}

\usepackage{hyperref}

\allowdisplaybreaks
\theoremstyle{plain}
\newtheorem{theorem}{Theorem}[section]
\newtheorem{proposition}[theorem]{Proposition}
\newtheorem{corollary}[theorem]{Corollary}
\newtheorem{lemma}[theorem]{Lemma}

\theoremstyle{definition}
\newtheorem{definition}[theorem]{Definition}
\newtheorem{remark}[theorem]{Remark}
\newtheorem*{acknowledgement}{Acknowledgement}

\newcommand{\CBbb}{\mathbb C}
\newcommand{\PBbb}{\mathbb P}
\newcommand{\Span}{{\rm Span}}

\begin{document}
\title[Biholomorphic rigidity and symmetry breaking for Bergman metrics]{Biholomorphic rigidity and symmetry breaking for domains with Bergman metric of constant positive holomorphic sectional curvature}
\author{Gordon Heier, Ravi Mudaliar, and Haotian Zuo}

\address{Department of Mathematics\\University of Houston\\4800 Calhoun Road\\ Houston, TX 77204\\USA}
\email{heier@math.uh.edu}\email{rmudalia@cougarnet.uh.edu}\email{hzuo@cougarnet.uh.edu}

\thanks{The first author was supported by a grant from the Simons Foundation (Grant Number 963755-GH)}



\begin{abstract}
Based on earlier work by Wiegerinck and Huang-Li, Hayashida-Kamimoto constructed unbounded non-smooth Reinhardt domains in \({\mathbb C}^n\), \(n\geq 3\), whose Bergman spaces are finite-dimensional and whose Bergman metrics have constant positive holomorphic sectional curvature \(2\). In this note, we prove a Sandwich Principle for the Hayashida-Kamimoto domains that allows us to construct, inside every prescribed
sandwich, uncountably many pairwise biholomorphically inequivalent smooth non-Reinhardt domains whose Bergman metrics also have constant positive holomorphic sectional curvature \(2\). We also classify biholomorphisms between domains sandwiched by the domains of Hayashida-Kamimoto with the help of an affine rigidity theorem in the spirit of Cartan rigidity, which should be of independent interest. For explicit angular perturbation families, we determine the
full automorphism group and obtain, in particular, an abundance of examples with trivial automorphism group.
\end{abstract}

\maketitle
\section{Introduction}

It is a fundamental fact in the theory of several complex variables that every bounded domain
\(D\subset\mathbb C^n\) carries a well-defined K\"ahler metric
associated with the complex Hilbert space \(A^2(D)\) of
\(L^2\)-integrable holomorphic functions on \(D\). This metric,
called the Bergman metric of \(D\), is independent of the choice of
an orthonormal basis of \(A^2(D)\). The monomials
\[
z_1^{\alpha_1}\cdots z_n^{\alpha_n},
\qquad
(\alpha_1,\ldots,\alpha_n)\in\mathbb N_0^n,
\]
form a linearly independent subset of \(A^2(D)\), but in general
they need not be orthogonal and need not form a basis. If \(D\) is
a bounded Reinhardt domain containing the origin, then the
monomials do form an orthogonal basis. Their normalization constants,
however, depend strongly on the geometry of \(D\), and the
curvature of the resulting Bergman metric is in general difficult
to determine. In particular, there is no reason to expect the
holomorphic sectional curvature to have a definite sign, although
Bergman \cite{Bergmann_1948} proved that the holomorphic sectional
curvature of a bounded domain is bounded above by \(2\).

The unit $n$-ball in $\mathbb C^n$ has a Bergman metric of constant negative holomorphic sectional curvature and represents the negative curvature case of the three complex space forms. On the other hand, Huang-Li \cite[Corollary 1.2]{Huang_Li} proved that no bounded domain can have a Bergman metric of constant positive holomorphic sectional curvature. Their proof uses Calabi’s rigidity theorem \cite{Calabi_1953}. Therefore, unbounded domains with finite-dimensional Bergman spaces provide the basic setting in which to look for Bergman metrics of constant positive holomorphic sectional curvature.  Based on Wiegerinck's construction
\cite{Wiegerinck}, Huang-Li \cite{Huang_Li} produced an unbounded smooth Reinhardt domain in
\(\mathbb C^2\) whose Bergman space is generated by the constant and
linear coordinate functions and whose Bergman metric has constant
positive holomorphic sectional curvature \(2\).
Hayashida-Kamimoto recently extended this construction to higher
dimensions \cite{HayashidaKamimoto_2026}, albeit with a focus on non-smooth domains.  In the relevant parameter
range, their domains have Bergman space
\(\operatorname{Span}\{1,z_1,\ldots,z_n\}\)
and Bergman metric of constant holomorphic sectional curvature \(2\).

However, these existence results do not address natural questions that arise in the large.  A
constant-curvature Bergman metric controls the local differential
geometry, but does not by itself classify the underlying domains. It is clearly interesting to determine when two such domains are biholomorphically
equivalent, which parameters are intrinsic, and how much holomorphic
symmetry is forced by the finite-dimensional Bergman space structure.  The
present paper addresses these questions in the given context and shows, in particular,
that constant positive Bergman curvature is compatible with extensive
symmetry breaking and nontrivial biholomorphic moduli.

The equivalence problem is closely related to the classical question
of whether biholomorphic equivalence between Reinhardt domains must be
algebraic, or monomial, in the coordinates.  The bounded and
hyperbolic settings were studied by Sunada \cite{Sunada_1978}, Shimizu \cite{Shimizu_1988}, and Kruzhilin \cite{Kruzhilin_1989}.  In dimension two,
Soldatkin \cite{Soldatkin} proved an unrestricted monomial-equivalence theorem, while no corresponding theorem is known in arbitrary dimension.  This distinction already appears in the
Hayashida-Kamimoto family: Yamamori \cite{Origin_rigidity} obtained the two-dimensional
parameter classification from Soldatkin's theorem, whereas the available higher-dimensional
linearity result did not yield the full equivalence classification in
\cite{HayashidaKamimoto_2026}.

This paper has two logically distinct layers.  The first is a
curvature-independent rigidity theory for finite-dimensional degree-complete Bergman spaces.  Its
central input is the following new Cartan-type linearity theorem.  This affine rigidity theorem, in
contrast with the standard Cartan linearity phenomenon, assumes no boundedness,
circular symmetry, or distinguished fixed point. It is our first main theorem in this paper and should be of interest independently of the study of sandwich domains. For an integer
\(q\geq1\), put
\[
\mathcal P_{\leq q}=
\operatorname{Span}\{z^\alpha:|\alpha|\leq q\}.
\]

\begin{theorem}[Affine rigidity]
\label{mthm1}
Let $n\geq 2$ and let \(\Omega_1,\Omega_2\subset\mathbb C^n\) be domains satisfying
\(A^2(\Omega_1)
=A^2(\Omega_2)
=\mathcal P_{\leq q}\)
for some \(q\geq1\). Then every biholomorphism
\(F:\Omega_1\to\Omega_2\) is affine, i.e., 
\(
F(z)=Az+b\)
for some \(A\in\operatorname{GL}(n,\mathbb C)\) and
\(b\in\mathbb C^n\).
\end{theorem}

The proof uses the projective geometry intrinsic to the Bergman space. Therefore, the conclusion is obtained from the structure of
the Bergman space itself, without imposing symmetry on the domains or
  the Bergman kernel. In the excluded single variable case \(n=1\), it is a well-known fact due to Wiegerinck \cite{Wiegerinck} that any non-zero Bergman space is infinite-dimensional. Therefore, domains $\Omega_1,\Omega_2$ as assumed do not exist and the theorem is vacuous in that case.

Only the case \(q=1\) is needed for the constant-positive-curvature
sandwich construction in this paper.  The Cartan-type rigidity theorem
itself, as well as the subsequent equivalence and automorphism results
that do not use curvature, continue to apply in the corresponding
higher-degree cases \(q\geq2\).  Those cases form an independent
rigidity component of the paper rather than an extension needed for
the curvature construction.  The higher-degree examples of
Bhat-Ganguly-Nandi-Xiao \cite{Bhat_Ganguly_Nandi_Xiao_2026} show that this generality is geometrically
natural rather than merely formal.

The second layer of the paper concerns the following family of non-smooth Reinhardt domains.  For \(a>0\) and \(t>0\), define
\[
D_{a,t}
=
\left\{
z\in\mathbb C^n:
\left||z_k|^2-|z_n|^2\right|
<
\frac{t}{(1+|z_n|^2)^a},
\quad k=1,\ldots,n-1
\right\}.
\]
The domains \(D_{a,1}\) are the domains defined in Hayashida-Kamimoto
\cite{HayashidaKamimoto_2026}.  Our
starting point is that the Bergman space is unchanged upon passing to
an arbitrary intermediate domain between two members of this family of reference domains.

\begin{theorem}[Sandwich Principle]\label{sandwich}
Let \(a>0\),  \(0<t_1\leq t_2\), \(n\geq 2\) and $D\subset\mathbb{C}^n$ be a domain satisfying  $D_{a,t_1}\subset D \subset D_{a,t_2}$ (we shall call it \emph{\(a\)-sandwichable}, see Definition \ref{def:a-sandwichable}). Then $A^2(D_{a,t_1})\cong A^2(D)\cong A^2(D_{a,t_2})$. Moreover, if $\frac{2}{n-1}<a\leq\frac{3}{n-1}$, then the Bergman map of $D$ is a holomorphic isometric embedding onto an open subset of $\mathbb{P}^n$, where $\mathbb{P}^n$ is equipped with the
Fubini-Study metric. In particular, the Bergman metric of \(D\) has constant holomorphic sectional curvature \(2\).
\end{theorem}

Note that no Reinhardt symmetry or boundary regularity is imposed on the
intermediate domain \(D\). Note also that, throughout this paper, the symbol $\subset$ permits equality.

Combining Theorem \ref{mthm1} with a linear preserver argument yields a complete
biholomorphic classification as follows. We refer to Remark \ref{n3_remark} for why the assumption of \(n\geq3\) is essential for us.
\begin{theorem}[Rigidity of sandwichable domains]
\label{biholclassification}
Let \(n\geq3\), and let
\(a_1,a_2>\frac{2}{n-1}.\) For \(j=1,2\), let \(\Omega_j\subset\mathbb C^n\) be an
\(a_j\)-sandwichable domain. If there is a biholomorphism
\(F:\Omega_1\longrightarrow\Omega_2\), then \(a_1=a_2\). Moreover, every biholomorphism \(F\) is of the form
\(F(z)=\rho DPz,\)
where \(\rho>0\), \(D\) is a diagonal unitary matrix, and \(P\) is a
permutation matrix.
In particular, every automorphism of an \(a\)-sandwichable domain is of the form \(F(z)= DPz\), i.e., has $\rho =1$.
\end{theorem}
In particular, biholomorphic equivalence among domains in a given sandwich
is reduced to an explicit monomial action.  Furthermore, in Theorem \ref{reference_domain_rigidity}, we will deduce the full parameter classification of the reference family:
\[
D_{a_1,t_1}
\ \text{is biholomorphic to}\
D_{a_2,t_2}
\quad\Longleftrightarrow\quad
(a_1,t_1)=(a_2,t_2)
\]
whenever \(a_1,a_2>2/(n-1)\).  Thus both parameters are genuine
biholomorphic invariants.

The Sandwich Principle also makes it possible to break symmetry while
preserving the relevant Bergman space.  We construct smooth angular
perturbations inside every prescribed sandwich, prove that they are
non-Reinhardt, and determine their full automorphism groups from the
perturbation data.  The resulting families exhibit both continuous
biholomorphic moduli and domains with no nontrivial holomorphic
symmetry as stated in the following theorem, which is the second main theorem of this paper.

\begin{theorem}[Abundance of domains]
\label{mthm2}
Let \(n\ge3\) and
\(\frac{2}{n-1}<a\le\frac{3}{n-1}.\)
For every \(0<t_1<t_2\), there exist uncountably many pairwise biholomorphically inequivalent smooth non-Reinhardt domains $D$ with
\(
D_{a,t_1}\subset D\subset D_{a,t_2}\), Bergman metric of constant
positive holomorphic sectional curvature \(2\), and trivial automorphism group.
\end{theorem}

The paper is organized as follows.
Section~\ref{definitions} gives basic definitions and records the finite-dimensional Bergman map
and the projective model used later.
Section~\ref{The Sandwich Principle} proves the Sandwich Principle, i.e., Theorem \ref{sandwich}.
Section~\ref{sec:rigidity} establishes
Theorem~\ref{mthm1} and classifies biholomorphisms in the
sandwich class, proving Theorem~\ref{biholclassification}.
Section~\ref{sec:general-perturbations} constructs the angular
perturbation families, and Section~\ref{sec:perturbation-aut}
determines their automorphism groups.
Section~\ref{sec:applications} derives the examples with trivial
automorphism group and proves Theorem~\ref{mthm2}.
Finally, Section~\ref{sec:reference-family} treats the reference
Reinhardt domains and determines their automorphism groups and parameter
rigidity.

\begin{acknowledgement}
We would like to thank Professor Xiaojun Huang for introducing us to the topic of unbounded domains with Bergman metric of constant positive holomorphic sectional curvature.
\end{acknowledgement}

\section{Definitions}\label{definitions}

\begin{definition}
A nonempty subset $D\subset \CBbb^n$ is called a {\it domain} if it is open and connected.
\end{definition}
Any domain $D\subset \CBbb^n$ admits a real-valued smooth, i.e., $C^\infty$, function $\gamma$ such that $D=\{\gamma <0\}$ and the topological boundary $\partial D$ satisfies $\partial D =\{\gamma =0\}$. We say that $D$ is a {\it smooth domain} if $\gamma$ can be chosen such that $d\gamma$ is non-zero everywhere on $\partial D$.
\begin{definition}\label{reinhardt}
A domain $D$ is called {\it Reinhardt} if $z=(z_1,\ldots,z_n)\in D$ implies 
\((e^{i\theta_1}z_1,\ldots,e^{i\theta_n}z_n)\in D\)
for all
\(0\leq \theta_j < 2\pi,~j=1,\ldots,n.\)
A domain $D$ is called {\it complete Reinhardt} if
\(z\in D\)
implies that
\((\mu_1 z_1,\ldots,\mu_n z_n)\in D\)
for all
\(\mu_j\in \mathbb C,~|\mu_j|\leq 1,~ j=1,\ldots,n.\)
\end{definition}

\begin{remark}
The above definition of a complete Reinhardt domain is the one from Krantz's book \cite{Krantz_FTSCV}. It appears that the papers \cite{Huang_Li} and \cite{HayashidaKamimoto_2026} work with a different definition of that notion. While the domains studied in those papers are clearly Reinhardt since only the absolute values of the coordinate functions occur in their defining functions, they are not complete Reinhardt in the sense of Definition \ref{reinhardt} for the following reason. The diagonal points $(1,1)$ and $(1,\ldots,1)$, respectively, are in the domain, while the points $(0,1)$ and $(0,\ldots,0,1)$, respectively, are not, which violates the definition of a complete Reinhardt domain with the choice $\mu_1=\ldots=\mu_{n-1}=0$, $\mu_n=1$. 
\end{remark}

\begin{definition}
For a domain $D$, we define its {\it Bergman space} to be
\[ 
A^2(D)=\left\{f:D\to \CBbb \text{ holomorphic} : \| f \|_D^2 = \int _D |f(z)|^2 dV < \infty \right\},
\]
where $dV$ is the Lebesgue measure on $\CBbb^n$.
\end{definition}

\begin{definition}\label{non_deg_bergman_definitions}
The Bergman space \(A^2(D)\) is said to
\begin{enumerate}
\item be \emph{base-point free} if, for every \(p\in D\), there exists
\(f\in A^2(D)\) such that \(f(p)\neq 0\);

\item \emph{separate points} if, for every \(p\neq q\) in \(D\),
there exists \(f\in A^2(D)\) such that
\(f(p)=0\) and \(f(q)\neq 0\);

\item \emph{separate holomorphic tangent directions} if, for every
\(p\in D\) and every nonzero \(X_p\in T_p^{1,0}D\), there exists
\(f\in A^2(D)\) such that
\(
f(p)=0 \text{ and } X_p(f)\neq 0.
\)

\end{enumerate}
\end{definition}

The Bergman space is a Hilbert subspace of the Hilbert space $L^2(D)$. 
Since $L^2(D)$ is separable, there is an orthonormal basis $\{\varphi_j\}_{j=1}^N$ of $A^2(D)$ with $N$ finite or infinite. The {\it Bergman kernel} of \(D\) is defined by
\[
K_D(z,w):=\sum_{j=1}^N \varphi_j(z)\overline{\varphi_j(w)}.
\]
This expression is independent of the chosen orthonormal basis.  The restriction of the 
Bergman kernel to the diagonal is given by
\[
K_ D(z,z)=\sum_{j=1}^N |\varphi_j(z)|^2 .
\]
In particular, \(K_D(z,z)>0\) if and only if there exists some
\(f\in A^2(D)\) such that \(f(z)\neq 0\). Thus base-point freeness
is precisely the condition ensuring that \(K_D(z,z)\) is positive
at every point of \( D\), so that \(\log K_D(z,z)\) is defined
everywhere.\par

The {\it Bergman metric} is then formally defined by
\[
g_{i\bar j}(p)
=
\frac{\partial^2}{\partial z_i\,\partial \overline{z_j}}
\log K_D(z,z)\bigg|_{z=p}\quad (1\leq i,j\leq n).
\]
However, this expression defines a genuine Hermitian metric only if the
associated Hermitian form $\omega$, called the {\it Bergman form}, is positive definite at every point. This
requires the Bergman space to separate holomorphic tangent directions:
for every \(p\in D\) and every nonzero
\(X_p\in T_p^{1,0} D\), there exists \(f\in A^2(D)\) such that
\[
f(p)=0
\qquad\text{and}\qquad
X_p(f)\neq 0.
\]
Consequently, the phrase ``the Bergman metric'' is justified as the name
of a genuine Hermitian (and then evidently K\"ahler) metric only when \(A^2( D)\) is base-point free
and separates holomorphic tangent directions. Without base-point
freeness, \(\log K_D(z,z)\) may fail to be defined; without
separation of tangent directions, the associated Hermitian form may be degenerate.\par

The preceding conditions also have a natural geometric interpretation in terms
of the Bergman map. Since this work only deals with the finite-dimensional case, we assume for this discussion that \(A^2(D)\) has finite dimension \(N\). Assuming that \(A^2(D)\)
is base-point free, the formula
\[
\beta_D:D\longrightarrow \mathbb{P}^{N-1},\qquad
z\longmapsto [\varphi_1(z):\ldots:\varphi_N(z)]
\]
defines a well-defined holomorphic map into projective space. 
The map \(\beta_D\) is independent of the chosen orthonormal basis up to a
unitary change of homogeneous coordinates. Moreover, the separation properties
above have the expected geometric meaning. The Bergman space separates points
precisely when \(\beta_D\) is injective, and it separates holomorphic tangent
directions precisely when \(\beta_D\) is an immersion.

The connection with the Bergman metric is given by the Fubini-Study form. If
\(\omega_{\mathrm{FS}}\) denotes the Fubini-Study form on projective space,
normalized in homogeneous coordinates by
\[
\omega_{\mathrm{FS}}
=
\sqrt{-1}\,\partial\overline{\partial}
\log \sum_{j=1}^N |Z_j|^2,
\]
then
\[
\beta_D^*\omega_{\mathrm{FS}}
=
\sqrt{-1}\,\partial\overline{\partial}
\log \sum_{j=1}^N |\varphi_j(z)|^2
=
\sqrt{-1}\,\partial\overline{\partial} \log K_D(z,z)
=
\omega_D .
\]
Thus, whenever the Bergman form is positive definite, the Bergman map is a
holomorphic isometric immersion of \(D\), equipped with its Bergman metric, into
projective space equipped with the Fubini-Study metric. If in addition
\(A^2(D)\) separates points, then this map is a holomorphic isometric injective immersion.
\begin{definition}\label{def:a-sandwichable}
Let \(a>0\). A domain \(\Omega\subset\mathbb C^n\) is called
\emph{\(a\)-sandwichable} if there exist \(0<t_1\leq t_2\) such that
\(D_{a,t_1}\subset\Omega\subset D_{a,t_2}.\) We say \(\Omega\) is \emph{sandwichable} if it is \(a\)-sandwichable for some \(a\).
\end{definition}

\section{The Sandwich Principle, Proof of Theorem \ref{sandwich}}\label{The Sandwich Principle}
Theorem \ref{sandwich} is based on the following explicit determination of the Bergman spaces of \(D_{a,1}\) for all values \(a>0\). 
\begin{theorem}[{\cite[Theorem 2.1]{HayashidaKamimoto_2026}}]\label{HK_integral_computation_theorem}
If \(a>\frac{1}{n-1}\), then the Bergman space of \(D_{a,1}\) is given by
\[
A^2(D_{a,1})
=
\operatorname{Span}\{z^p : p\in \mathbb N_0^n,\ |p|<(n-1)a-1\}.
\]
Furthermore, the dimension of \(A^2(D_{a,1})\) can be expressed as
\[
\dim A^2(D_{a,1})
=
\begin{cases}
0
& \text{if } 0<a\leq \dfrac{1}{n-1}, \\[6pt]
\binom{n+k}{k}
& \text{if } \dfrac{k+1}{n-1}<a\leq \dfrac{k+2}{n-1}
\quad (k\in \mathbb N_0).
\end{cases}
\]
\end{theorem}
To obtain our Sandwich Principle, we first establish that the Bergman spaces of the domains $D_{a,t}$ are independent of the value of the parameter $t > 0$ for a fixed value of $a>0$.
\begin{lemma}\label{change_of_var}
Let \(a>0\) and \(0<t_1\leq t_2\). For every multi-index
\(\alpha=(\alpha_1,\ldots,\alpha_n)\in\mathbb N_0^n\), if
\(z^\alpha\in A^2(D_{a,t_1})\), then
\(z^\alpha\in A^2(D_{a,t_2}),\)
where \(z^\alpha=z_1^{\alpha_1}\cdots z_n^{\alpha_n}\).
\end{lemma}

\begin{proof}
Set \( \lambda=\frac{t_2}{t_1}\geq 1\). By introducing polar coordinates, we have
\[
\|z^\alpha\|_{A^2(D_{a,t_2})}^2
=(2\pi)^n\int_0^\infty
r_n^{2\alpha_n+1}\prod_{k=1}^{n-1}
\left(\int_{\left|r_k^2-r_n^2\right|<\frac{t_2}{(1+r_n^2)^a}}r_k^{2\alpha_k+1}\,dr_k
\right)dr_n.
\]
We now perform the change of variables
\(
r_j=\sqrt{\lambda}\,R_j\), for \(j=1,\ldots,n\), making the above equal to 
 \begin{align*}
    &\lambda^{|\alpha|+n} (2\pi)^n \int_0^\infty R_n^{2\alpha_n+1} \prod_{k=1}^{n-1}\int_{\left| R_k^2-R_n^2\right|<\frac{t_1}{\bigl(1+\lambda R_n^2\bigr)^a}}R_k^{2\alpha_k+1}dR_kdR_n\\
    \leq~& \lambda^{|\alpha|+n} (2\pi)^n \int_0^\infty R_n^{2\alpha_n+1} \prod_{k=1}^{n-1}\int_{\left| R_k^2-R_n^2\right|<\frac{t_1}{\bigl(1+R_n^2\bigr)^a}} R_k^{2\alpha_k+1}dR_kdR_n\\
    =~& \lambda^{|\alpha|+n}||z^\alpha||_{A^2(D_{a,t_1})}^2<\infty.
\end{align*}
\end{proof}

Since \(D_{a,t_1}\subset D_{a,t_2}\) for \(0< t_1\leq t_2\), Lemma \ref{change_of_var} immediately yields
\begin{corollary}\label{isom}
Let \(a>0\) and \(0< t_1\leq t_2\). The natural restriction map induces an isomorphism 
\[\iota: A^2(D_{a,t_2})\to  A^2(D_{a,t_1}).\]
\end{corollary}
An important point here is that when the Bergman spaces are non-zero, $\iota$ is not an isometry unless $t_1=t_2$. However, since both \(D_{a,t_1}\) and \(D_{a,t_2}\) are Reinhardt, appropriate monomials from Theorem \ref{HK_integral_computation_theorem} form an orthogonal (but not orthonormal) basis for both vector spaces. An immediate corollary is the following, obtained by setting either $t_1$ or $t_2$ equal to $1$.
\begin{corollary}\label{isom_cor}
    If $\frac{2}{n-1}<a\leq\frac{3}{n-1}$ and $t>0$, then  $A^2(D_{a,t})={\Span}\{1,z_1,\ldots,z_n\}.$
\end{corollary}
The above now allows us to prove our Sandwich Principle. 
\begin{proof}[Proof of Theorem \ref{sandwich}]
By the monotonicity of the integrals of non-negative functions with respect to nested domains of integration we know that 
\[A^2(D_{a,t_2})\subset A^2(D)\subset A^2(D_{a,t_1}).\]
An application of Corollary \ref{isom} shows the equality \(A^2(D_{a,t_1})=A^2(D_{a,t_2})\) and thus equality everywhere in the above chain of inclusion.  Moreover, if $\frac{2}{n-1}<a\leq\frac{3}{n-1}$, then Corollary \ref{isom_cor} yields
\[A^2(D)={\Span}\{1,z_1,\ldots,z_n\}.\]
Due to the absence of an assumption of Reinhardt symmetry for $D$, the monomials $1,z_1,\ldots,z_n$ will in general not be orthogonal to each other and the Bergman map of $D$ will therefore differ from the map
\begin{equation}\label{non_bergman_map}
D\to \PBbb^n,\quad (z_1,\ldots,z_n)\mapsto [1:z_1:\ldots:z_n].
\end{equation}
However, by standard linear algebra, the map \eqref{non_bergman_map} will differ from the Bergman map at most by a projective linear transformation of $\PBbb^n$. This is enough to prove that the Bergman map satisfies all three properties in Definition \ref{non_deg_bergman_definitions}. It is therefore a holomorphic embedding onto an open subset of $\mathbb{P}^n$ and also an isometry with respect to the Bergman metric on $D$ and the Fubini-Study metric on $\mathbb{P}^n$. 
\end{proof}
\begin{remark}
Note that the projective linear transformation of $\PBbb^n$ arising at the end of the proof of Theorem \ref{sandwich} will in general not be an isometry with respect to the Fubini-Study metric of $\PBbb^n$, i.e., not be an element of \( \operatorname{PU}(n+1,\mathbb C)\) . However, we still have enough information for our Sandwich Principle to hold true, despite the absence of any symmetry assumption for $D$.
\end{remark}

\section{Biholomorphic rigidity, Proof of Theorems \ref{mthm1} and \ref{biholclassification}}\label{sec:rigidity}
In this section we begin our study of rigidity and prove Theorems \ref{mthm1} and \ref{biholclassification}.\par
First, we record the following additional corollary to Theorem \ref{HK_integral_computation_theorem} and Corollary \ref{isom}.
\begin{corollary}\label{BergmanspaceofDat}
    For \(n\geq 2\), \(a>\frac{1}{n-1}\) and \(t>0\),  the Bergman space of \(D_{a,t}\) is given by
\[
A^2(D_{a,t})=
\operatorname{Span}\{z^p : p\in \mathbb N_0^n,\ |p|<(n-1)a-1\}.
\]
\end{corollary}

Now we establish our general affine rigidity theorem, i.e., Theorem \ref{mthm1}
\begin{proof}[Proof of Theorem \ref{mthm1}]
Let
\[
N+1=\#\{\alpha\in\mathbb Z_{\ge0}^n:\ |\alpha|\le q\}=\binom{n+q}{q}.
\]
Choose an orthonormal basis
\(e_{i,0},\ldots,e_{i,N}\)
of \(A^2(\Omega_i)\), and put
\[
e_i(z):=\begin{pmatrix}
e_{i,0}(z), \ldots,
e_{i,N}(z)
\end{pmatrix}^T,
\qquad i=1,2.
\]
The associated Bergman maps are
\[
\Phi_i:\Omega_i\longrightarrow\mathbb P^N,
\qquad
\Phi_i(z)=[e_i(z)],
\qquad i=1,2.
\]
These maps are well-defined, because the constant function \(1\) belongs to
\(A^2(\Omega_i)\) and therefore not all elements of \(A^2(\Omega_i)\) can vanish at
any given point of \(\Omega_i\).

Now order the monomials \(z^\alpha\), \(|\alpha|\le q\), and write
\[
e^\prime(z):=\bigl(z^\alpha\bigr)_{|\alpha|\le q}.
\]
Since
\[
\{e_{i,0},\ldots,e_{i,N}\},\ i=1,2,
\text{ and }
\{z^\alpha:\ |\alpha|\le q\}
\]
are bases of the same vector space \(A^2(\Omega_i)=\mathcal P_{\leq q}\),
there exist matrices
\[
S_i\in GL(N+1,\mathbb C),\qquad i=1,2,
\]
such that
\[
e_i(z)=S_i e^\prime (z),
\qquad i=1,2.
\]
Consequently,
\[
\Phi_i(z)=[S_i e^\prime(z)]=[S_i]\,[e^\prime(z)],
\qquad i=1,2,
\]
where \([S_i]\) denotes the projective linear transformation of \(\mathbb P^N\) induced by the matrix \(S_i\).
Let
\[
\iota:\mathbb C^n\longrightarrow\mathbb P^n,
\qquad
\iota(z)=[1:z_1:\cdots:z_n]
\]
be the canonical embedding map, and let
\[
\nu_q:\mathbb P^n\longrightarrow\mathbb P^N
\]
be the \(q\)-th Veronese embedding, written on the affine chart \(x_0=1\)
as
\[
\nu_q([1:z_1:\cdots:z_n])
=[z^\alpha]_{|\alpha|\leq q}=[e^\prime(z)].
\]
Therefore
\[
\Phi_i=[S_i]\circ\nu_q\circ\iota,
\qquad i=1,2.
\]

Let
\[
F:\Omega_1\longrightarrow\Omega_2
\]
be a biholomorphism, and write
\[
J_F(z):=\det F'(z).
\]
The biholomorphic change-of-variables formula shows that the operator
\[
U_F:A^2(\Omega_2)\longrightarrow A^2(\Omega_1),
\qquad
U_F f(z):=J_F(z)f(F(z)),
\]
is unitary. Indeed,
\[
\begin{aligned}
\|U_Ff\|_{L^2(\Omega_1)}^2
&=\int_{\Omega_1} |J_F(z)|^2|f(F(z))|^2\,dV(z)\\
&=\int_{\Omega_2} |f(w)|^2\,dV(w)
=\|f\|_{L^2(\Omega_2)}^2.
\end{aligned}
\]
Thus there exists a unitary matrix \(M_F\in U(N+1,\mathbb C)\) such that
\[
J_F(z)e_2(F(z))=M_F e_1(z).
\]
Since \(J_F(z)\neq0\), projectivizing gives
\begin{equation}\label{phibiholo}
    \Phi_2(F(z))
=[e_2(F(z))]=[J_F(z)e_2(F(z))]
=[M_F e_1(z)]
=[M_F]\Phi_1(z).
\end{equation}
Note that equation (\ref{phibiholo}) means
\[
[M_F]\bigl(\Phi_1(\Omega_1)\bigr)
=\Phi_2(F(\Omega_1))
=\Phi_2(\Omega_2).
\]
Thus the projective linear transformation \([M_F]\) sends the Bergman image
\(\Phi_1(\Omega_1)\) onto the Bergman image \(\Phi_2(\Omega_2)\).

We now identify the Zariski closures of \(\Phi_i(\Omega_i)\). 
Since \(\Omega_i\) is a nonempty Euclidean open subset of
\(\mathbb C^n\), we have
\[
\overline{\iota(\Omega_i)}^{\,\mathrm{Zar}}=\mathbb P^n.
\]
Since \([S_i]\circ\nu_q:\mathbb P^n\to [S_i](\nu_q(\mathbb P^n))\) is an
isomorphism onto its closed image, we obtain
\[
\overline{\Phi_i(\Omega_i)}^{\,\mathrm{Zar}}
=[S_i](\nu_q(\mathbb P^n))=:X_i,
\qquad i=1,2.
\]
Since \(M_F\) is a homeomorphism and
\([M_F]\Phi_1(\Omega_1)=\Phi_2(\Omega_2),\)
we have
\[
[M_F](X_1)
=[M_F]\left(\overline{\Phi_1(\Omega_1)}^{\,\mathrm{Zar}}\right)
=\overline{\left([M_F]\Phi_1(\Omega_1)\right)}^{\,\mathrm{Zar}}
=\overline{\Phi_2(\Omega_2)}^{\,\mathrm{Zar}}
=X_2.
\]

By definition of \(X_i\), the map
\[
[S_i]\circ\nu_q:\mathbb P^n\longrightarrow X_i
\]
is an isomorphism of projective varieties. Therefore
\[
g:=
([S_2]\circ\nu_q)^{-1}
\circ [M_F]\left.\right|_{X_1}
\circ ([S_1]\circ\nu_q)
\]
is an automorphism of \(\mathbb P^n\). Hence
\[
g\in \operatorname{Aut}(\mathbb P^n)
=\operatorname{PGL}(n+1,\mathbb C),
\]
and
\[
([S_2]\circ\nu_q)\circ g
=[M_F]\left.\right|_{X_1}\circ ([S_1]\circ\nu_q).
\]
Consider the diagram 
\begin{equation}\label{commdiagrambiholo}
    \begin{tikzcd}
{\Omega_1} \arrow[r, "F"] \arrow[d, "\iota"']      & {\Omega_2} \arrow[d, "\iota"]          \\
\mathbb P^n \arrow[r, "g"] \arrow[d, "{[S_1]}\circ\nu_q"'] & \mathbb P^n \arrow[d, "{[S_2]}\circ\nu_q"]     \\
X_1 \arrow[r, "{\left.[M_F]\right|_{X_1}}"']               & X_2 .                                 
\end{tikzcd}
\end{equation}
Since
\[
\begin{aligned}
({[S_2]}\circ\nu_q)(\iota(F(z)))
&=
\Phi_2(F(z))  \\
&=
[M_F](\Phi_1(z)) \\
&=
[M_F]\bigl(({[S_1]}\circ\nu_q)(\iota(z))\bigr) \\
&=
({[S_2]}\circ\nu_q)\bigl(g(\iota(z))\bigr)
\end{aligned}
\]
and 
\({[S_2]}\circ\nu_q\)
is injective, it follows that
\begin{equation}\label{if=gi-biholo}
    \iota(F(z))=g(\iota(z))
\qquad
\text{for every } z\in \Omega_1.
\end{equation}

In other words, the diagram (\ref{commdiagrambiholo})
commutes.

Since
\(g\in \operatorname{PGL}(n+1,\mathbb C),\)
choose a matrix representative
\[G=
\begin{pmatrix}
d & c^T\\
b & A
\end{pmatrix}
\in GL(n+1,\mathbb C),
\]
where
\[
d\in\mathbb C,\quad c,b\in\mathbb C^n,\quad \text{and}\quad A\in M_n(\mathbb C).
\]
Then
\[
g([x_0:x])=
[d x_0+c^Tx:\, b x_0+Ax].
\]
Putting \(x_0=1\) and \(x=z\), equation (\ref{if=gi-biholo}) gives
\[
[1:F(z)]=[d+c^Tz:\, b+Az].
\]
Consequently, wherever \(d+c^Tz\neq0\),
\begin{equation}\label{linearfractional-biholo}
   F(z)=\frac{Az+b}{d+c^Tz}. 
\end{equation}
We now show that the denominator in (\ref{linearfractional-biholo}) is constant.
Set
\(
\ell(z):=d+c^Tz,
~
N(z):=Az+b,
\)
then
\[
F(z)=\frac{N(z)}{\ell(z)}.
\]

First note that
\[
\ell(z)\neq 0
\qquad
\text{for every }z\in\Omega_1.
\]
Indeed, the projective identity obtained above is
\[
[1:F(z)]=[\ell(z):N(z)].
\]
Since the first homogeneous coordinate of \([1:F(z)]\) is nonzero,
the first homogeneous coordinate \(\ell(z)\) of
\([\ell(z):N(z)]\) must also be nonzero.

Differentiating the fractional linear expression gives
\[
F'(z)
=
\frac{\ell(z)A-N(z)c^T}{\ell(z)^2}
=
\frac{A-F(z)c^T}{\ell(z)}.
\]
We claim that
\[
\det G
=
\ell(z)\det\bigl(A-F(z)c^T\bigr).
\]
Indeed, adding \(z_j\) times the \((j+1)\)-st column of \(G\)
to its first column, for \(j=1,\ldots,n\), gives
\[
\det G
=
\det
\begin{pmatrix}
\ell(z) & c^T\\
N(z) & A
\end{pmatrix}.
\]
Next, for each \(i=1,\ldots,n\), subtract
\(F_i(z)\) times the first row from the \((i+1)\)-st row.
Since \(N(z)=\ell(z)F(z)\), this yields
\[
\det G
=
\det
\begin{pmatrix}
\ell(z) & c^T\\
0 & A-F(z)c^T
\end{pmatrix}
=
\ell(z)\det\bigl(A-F(z)c^T\bigr).
\]
Consequently, \begin{equation}\label{J_Fz}
    J_F(z):=
\det F'(z)
=\frac{\det\bigl(A-F(z)c^T\bigr)}{\ell(z)^n}
=\frac{\det G}{\ell(z)^{n+1}}.
\end{equation}

On the other hand, since \(1\in A^2(\Omega_2)\), the unitary operator
\(U_F\) satisfies
\[
J_F(z)=U_F(1)\in A^2(\Omega_1)=\mathcal P_{\leq q}.
\]
Hence there exists a polynomial
\(p\in\mathcal P_{\leq q}\)
such that
\(
J_F(z)=p(z)
\text{ for every }z\in\Omega_1.
\)
Combining this identity with (\ref{J_Fz}), we obtain
\(
p(z)\ell(z)^{n+1}=\det G
\text{ for every }z\in\Omega_1.
\)
Both sides are polynomials on \(\mathbb C^n\). Since \(\Omega_1\)
is a nonempty open subset of \(\mathbb C^n\), the identity theorem
implies that
\begin{equation}\label{extension}
    p(z)\ell(z)^{n+1}=\det G
\qquad
\text{for every }z\in\mathbb C^n.
\end{equation}

Suppose that \(c\neq 0\). Then the nonconstant affine linear
function
\(\ell(z)=d+c^Tz\)
has a zero at some point \(z_0\in\mathbb C^n\). Evaluating
(\ref{extension}) at \(z_0\) gives
\(\det G=0,\)
contrary to \(G\in\operatorname{GL}(n+1,\mathbb C)\).
Therefore
\(c=0.\)

It follows that
\[
G=
\begin{pmatrix}
d & 0\\
b & A
\end{pmatrix}.
\]
Since \(G\) is invertible, we have
\(
d\neq 0
\text{ and }
A\in\operatorname{GL}(n,\mathbb C).
\)
After dividing the matrix by \(d\), we may write
\[
F(z)=Az+b,
\qquad
A\in GL(n,\mathbb C),\quad b\in\mathbb C^n.
\]
In other words, \(F\) is the restriction of an affine automorphism of
\(\mathbb C^n\).
\end{proof}
\begin{remark}
    A related rigidity statement appears in the recent work of
Bhat--Ganguly--Nandi--Xiao \cite{Bhat_Ganguly_Nandi_Xiao_2026}, but the
hypotheses and proof mechanism are different.  Their
Proposition~4.3 starts from the prescribed Fubini--Study-type kernel
\(C(1+\|z\|^2)^q\): Calabi rigidity first gives a projective unitary
map, and the special kernel formula is then used again to remove the
projective denominator and translation.  Our theorem begins only with
the vector space identity \(A^2(\Omega)=\mathcal P_{\leq q}\); the
Veronese model and the polynomial behavior of the Jacobian replace the
special-kernel argument.  The two results are not in a
containment relation: their proposition treats kernel-preserving maps
between subdomains and obtains unitary linearity, whereas our theorem
treats global biholomorphisms and obtains affine linearity.  In the
global setting, however, our result recovers the affine-rigidity step
under the weaker polynomial-space hypothesis.
\end{remark}\begin{proposition}\label{biholrigidity}
Let \(n\ge2\), let \(q\ge1\), and let
\(\Omega_1,\Omega_2\subset\mathbb C^n\) be domains satisfying
\(
A^2(\Omega_1)=A^2(\Omega_2)=\mathcal P_{\leq q}.
\)
Assume that, for some positive constants \(a_1,a_2,s,t\),
\(
D_{a_1,s}\subset \Omega_1,
\text{ and }
\Omega_2\subset D_{a_2,t}.\)
Then every biholomorphism
\(
F:\Omega_1\longrightarrow\Omega_2\)
is linear, i.e., there exists \(A\in GL(n,\mathbb C)\) such that
\(F(z)=Az.\)
\end{proposition}

\begin{proof}
By Theorem \ref{mthm1}, 
we only need to prove that \(b=0\).

Write \(R_j\) for the \(j\)-th row of \(A\), so that
\(
(Az)_j=R_jz.\)
Let \(u\) be a point on the standard $n$-torus \(\mathbb T^n\subset \mathbb C^n\). Then for every real number \(\lambda>0\),
\(\lambda u\in D_{a_1,s}\subset\Omega_1.\)
Since \(F(\Omega_1)=\Omega_2\subset D_{a_2,t}\), we have
\[
F(\lambda u)=\lambda Au+b\in D_{a_2,t}.
\]
Therefore, for each \(k=1,\ldots,n-1\),
\begin{equation}\label{sandwich-polynomial-estimate}
\left|
|\lambda R_ku+b_k|^2-|\lambda R_nu+b_n|^2
\right|
<
\frac{t}{(1+|\lambda R_nu+b_n|^2)^{a_2}}.
\end{equation}

Consider first the set
\[
U:=\{u\in\mathbb T^n:\ R_nu\neq0\}.
\]
This set is dense in \(\mathbb T^n\), because \(R_n\neq0\) and a nonzero linear form cannot vanish
on a nonempty open subset of \(\mathbb T^n\). For such \(u\), the right-hand side of
\eqref{sandwich-polynomial-estimate} tends to \(0\) as \(\lambda\to+\infty\).

The expression inside the absolute value on the left-hand side of
\eqref{sandwich-polynomial-estimate} is the real polynomial
\[
\begin{aligned}
p_{k,u}(\lambda)
&:=
|\lambda R_ku+b_k|^2-|\lambda R_nu+b_n|^2\\
&=
\lambda^2\left(|R_ku|^2-|R_nu|^2\right)
+2\lambda\operatorname{Re}\left((R_ku)\overline{b_k}-(R_nu)\overline{b_n}\right)
+|b_k|^2-|b_n|^2.
\end{aligned}
\]
Since
\(
|p_{k,u}(\lambda)|\longrightarrow0
\text{ as }\lambda\to+\infty,
\)
the polynomial \(p_{k,u}\) must be identically zero. In particular,
\begin{equation}\label{linear-coefficient-vanishing}
\operatorname{Re}\left((R_ku)\overline{b_k}-(R_nu)\overline{b_n}\right)=0.
\end{equation}
By density and continuity, \eqref{linear-coefficient-vanishing} holds for every
\(u\in\mathbb T^n\).

Write
\[
R_k=(a_{k1},\ldots,a_{kn}),
\qquad
R_n=(a_{n1},\ldots,a_{nn}).
\]
Then \eqref{linear-coefficient-vanishing} says
\[
\operatorname{Re}
\sum_{j=1}^n
\left(a_{kj}\overline{b_k}-a_{nj}\overline{b_n}\right)u_j
=0
\qquad
\text{for every }u=(u_1,\ldots,u_n)\in\mathbb T^n.
\]
By uniqueness of Fourier coefficients on \(\mathbb T^n\), we get
\[
a_{kj}\overline{b_k}=a_{nj}\overline{b_n},
\qquad
j=1,\ldots,n.
\]
Equivalently,
\[
\overline{b_k}R_k=\overline{b_n}R_n.
\]
Since \(A\in GL(n,\mathbb C)\), its rows are linearly independent. Hence \(R_k\) and \(R_n\)
are linearly independent, and therefore
\(b_k=b_n=0.\)
This holds for every \(k=1,\ldots,n-1\). Hence
\(b=0.\)

Consequently
\(
F(z)=Az
\)
with \(A\in GL(n,\mathbb C)\). Thus every biholomorphism
\(F:\Omega_1\longrightarrow\Omega_2\) is linear.
\end{proof}
To describe automorphism groups more precisely, we resort to the concept of horizon cones.
\begin{definition}\cite{Variational Analysis}
    For an arbitrary set
\(E\subset\mathbb C^n\), the cone
\[
E^\infty
:=
\left\{
v\in\mathbb C^n:
\begin{array}{l}
\text{there exist }z^\nu\in E\text{ and }\lambda_\nu>0\text{ such that}\\[2mm]\lambda_\nu\to0
\text{ and }\lambda_\nu z^\nu\to v
\end{array}
\right\}
\] is called the horizon cone of \(E\).
\end{definition}
This cone records the directions in which \(E\) extends to infinity.

Now set
\(\mathcal C:=\left\{
v\in\mathbb C^n:|v_1|=\cdots=|v_n|\right\}.\)

\begin{lemma}\label{horizoncone}
For every \(a,t>0\),
\(D_{a,t}^{\infty}=\mathcal C.\)
\end{lemma}

\begin{proof}
Let \(v\in D_{a,t}^{\infty}\). By definition, there exist
\(z^\nu\in D_{a,t}\) and positive numbers \(\lambda_\nu\to0\) such that
\(\lambda_\nu z^\nu\longrightarrow v.\)
Since \(z^\nu\in D_{a,t}\), for every \(k=1,\ldots,n-1\),
\[\left|
|z_k^\nu|^2-|z_n^\nu|^2\right|<
\frac{t}{(1+|z_n^\nu|^2)^a}.
\]
Multiplying this inequality by \(\lambda_\nu^2\), we obtain
\[\left||\lambda_\nu z_k^\nu|^2-|\lambda_\nu z_n^\nu|^2
\right|<\frac{t\lambda_\nu^2}
{(1+|z_n^\nu|^2)^a}\leq t\lambda_\nu^2.
\]
Letting \(\nu\to\infty\) gives
\(|v_k|^2=|v_n|^2,~ k=1,\ldots,n-1.\)
Hence
\(|v_1|=\cdots=|v_n|,\)
and therefore \(v\in\mathcal C\). This proves that
\(D_{a,t}^{\infty}\subset\mathcal C.\)

Conversely, let \(v\in\mathcal C\). Define
\(z^\nu:=\nu v,\text{ and }\lambda_\nu:=\frac1\nu.\)
Since \(|v_k|=|v_n|\)
for every \(k=1,\ldots,n-1\), we have
\[|z_k^\nu|^2-|z_n^\nu|^2=\nu^2\bigl(|v_k|^2-|v_n|^2\bigr)
=0.
\]
It follows that
\(z^\nu\in D_{a,t}\)
for every \(\nu\), while
\(\lambda_\nu z^\nu=v.\)
Thus \(v\in D_{a,t}^{\infty}\), and consequently
\(
\mathcal C\subset D_{a,t}^{\infty}.
\)
\end{proof}

\begin{lemma}\label{horizonlinear}
Let \(E\subset\mathbb C^n\), and let
\(A\in GL(n,\mathbb C)\). Then
\((AE)^\infty=A(E^\infty).\)
\end{lemma}

\begin{proof}
Suppose first that \(v\in E^\infty\). Then there exist
\(z^\nu\in E\) and \(\lambda_\nu\to0\) such that
\(
\lambda_\nu z^\nu\to v.
\)
Since \(Az^\nu\in AE\), we have
\[
\lambda_\nu Az^\nu=A(\lambda_\nu z^\nu)\longrightarrow Av.
\]
Thus
\(
Av\in(AE)^\infty,
\)
and hence
\(
A(E^\infty)\subset(AE)^\infty.
\)

Applying the same argument to \(A^{-1}\) and the set \(AE\), we obtain
\(A^{-1}\bigl((AE)^\infty\bigr)\subset
E^\infty.
\)
Multiplying by \(A\) gives
\((AE)^\infty\subset A(E^\infty).\)
Therefore
\((AE)^\infty=A(E^\infty).\)
\end{proof}
Now let \(\Omega_1,\Omega_2\) satisfy the hypotheses of
Theorem~\ref{biholclassification}, and let
\(A:\Omega_1\to\Omega_2\) be the linear biholomorphism obtained from
Proposition~\ref{biholrigidity}. By monotonicity of horizon cones and
Lemma~\ref{horizoncone},
\[
\Omega_1^\infty=\mathcal C=\Omega_2^\infty.
\]
Since \((A\Omega_1)^\infty=A(\Omega_1^\infty)\), it follows that
\[
A(\mathcal C)=\mathcal C.
\]
The classification therefore reduces to the following linear preserver
problem.
\begin{lemma}\label{keylemma}
Let \(n\geq3\), and let \(A\in GL(n,\mathbb C)\) satisfy
\(A\left(\mathcal C\right)=\mathcal C.\)
Then
\(A=\rho DP,\)
where \(\rho>0\), \(D\) is a diagonal unitary matrix, and \(P\) is a
permutation matrix.
\end{lemma}
\begin{proof}
Write
\[
A=
\begin{pmatrix}
a^{(1)}\\
\vdots\\
a^{(n)}
\end{pmatrix},
\qquad
a^{(r)}
=
\bigl(a_1^{(r)},\ldots,a_n^{(r)}\bigr),
\]
where \(a^{(r)}\) denotes the \(r\)-th row of \(A\).

Let
\[
\mathbb T^n=\left\{\zeta=(\zeta_1,\ldots,\zeta_n)\in\mathbb C^n:
|\zeta_1|=\cdots=|\zeta_n|=1\right\}.
\]
Since \(\mathbb T^n\subset\mathcal C\) and \(A(\mathcal C)=\mathcal C\),
the coordinates of \(A\zeta\) have equal modulus for every
\(\zeta\in\mathbb T^n\). Thus, for every pair \(r,s\in\{1,\ldots,n\}\),
\[\left|
\sum_{j=1}^n a_j^{(r)}\zeta_j
\right|^2=\left|\sum_{j=1}^n a_j^{(s)}\zeta_j
\right|^2\qquad
\text{for all }\zeta\in\mathbb T^n.
\]
Expanding both sides gives
\[
\sum_{j=1}^n |a_j^{(r)}|^2+
\sum_{\substack{j,k=1\\ j\neq k}}^n
a_j^{(r)}\overline{a_k^{(r)}}\,
\zeta_j\overline{\zeta_k}=
\sum_{j=1}^n |a_j^{(s)}|^2+
\sum_{\substack{j,k=1\\ j\neq k}}^n
a_j^{(s)}\overline{a_k^{(s)}}\,
\zeta_j\overline{\zeta_k}.
\]
 Comparing Fourier
coefficients therefore yields
\begin{equation}\label{eq:off-diagonal-coefficients}
a_j^{(r)}\overline{a_k^{(r)}}=
a_j^{(s)}\overline{a_k^{(s)}},
\qquad
j\neq k,
\end{equation}
and comparison of the constant terms gives
\begin{equation}\label{eq:row-norms}
\sum_{j=1}^n |a_j^{(r)}|^2=
\sum_{j=1}^n |a_j^{(s)}|^2.
\end{equation}

We claim that every off-diagonal coefficient in
\eqref{eq:off-diagonal-coefficients} must vanish. Suppose, to the
contrary, that there exist distinct indices \(p,q\) such that
\(a_p^{(r)}\overline{a_q^{(r)}}\neq 0 \)
for some row \(r\). By \eqref{eq:off-diagonal-coefficients}, the value of
this product is independent of the row, and hence
\(a_p^{(s)}\neq 0,~a_q^{(s)}\neq 0\)
for every \(s=1,\ldots,n\).

We first show that all rows of \(A\) have the same support. Fix an index
\(\ell\not = p\). If
\(a_\ell^{(r_0)}=0 \)
for some row \(r_0\), then
\(a_p^{(r_0)}\overline{a_\ell^{(r_0)}}=0.\)
By \eqref{eq:off-diagonal-coefficients},
\(a_p^{(s)}\overline{a_\ell^{(s)}}=0\)
for every \(s\). Since \(a_p^{(s)}\neq 0\), it follows that
\(a_\ell^{(s)}=0\)
for every \(s\). Conversely, if \(a_\ell^{(r_0)}\neq 0\) for some row
\(r_0\), then
\(a_p^{(r_0)}\overline{a_\ell^{(r_0)}}\neq 0,\)
and hence \eqref{eq:off-diagonal-coefficients} implies
\(a_\ell^{(s)}\neq 0\)
for every \(s\). Thus all rows have a common support, say
\(S\subset\{1,\ldots,n\}.\)
Since \(p,q\in S\), we have \(|S|\geq 2\).

Suppose first that \(|S|\geq 3\). Fix two rows \(r\) and \(s\), and for
each \(j\in S\) define
\(b_j:=a_j^{(s)}/{a_j^{(r)}}.\)
For distinct \(j,k\in S\), equation
\eqref{eq:off-diagonal-coefficients} gives
\(a_j^{(s)}\overline{a_k^{(s)}}=
a_j^{(r)}\overline{a_k^{(r)}},
\)
and hence
\(b_j\overline{b_k}=1.\)
Choose three distinct indices \(j,k,\ell\in S\). From
\(
b_j\overline{b_k}=1\text{ and }
b_j\overline{b_\ell}=1
\)
we obtain \(b_k=b_\ell\). Similarly, using
\(
b_k\overline{b_\ell}=1
\text{ and } b_j\overline{b_\ell}=1,
\)
we obtain \(b_j=b_k\). It follows that all \(b_j\), \(j\in S\), are
equal to a common scalar \(b\). Moreover,
\(b\overline b=1,\)
so \(|b|=1\). Therefore
\(a^{(s)}=b\,a^{(r)}.
\)
Since \(r\) and \(s\) were arbitrary, all rows of \(A\) are proportional.
This contradicts the invertibility of \(A\).

Suppose now that \(|S|=2\). Then every row of \(A\) is supported in the
same two coordinate positions. Hence the row space of \(A\) has dimension
at most \(2\), so
\(\operatorname{rank}A\leq 2<n,\)
again contradicting \(A\in GL(n,\mathbb C)\).

These contradictions show that
\(a_j^{(r)}\overline{a_k^{(r)}}=0
\text{ for every }r\text{ and every }j\neq k.
\)
Consequently, each row of \(A\) contains at most one nonzero entry.
Since \(A\) is invertible, every row contains exactly one nonzero entry.
Moreover, the nonzero entries must occur in distinct columns; otherwise
two rows would be supported in the same column and would therefore be
linearly dependent. Thus every column also contains exactly one nonzero
entry, and \(A\) is a monomial matrix.

Finally, by \eqref{eq:row-norms}, all rows of \(A\) have the same
Euclidean norm. Since each row contains exactly one nonzero entry, all
nonzero entries of \(A\) have the same modulus, say \(\rho>0\). Hence
\(A=\rho DP,\)
where \(P\) is a permutation matrix and \(D\) is a diagonal matrix whose
diagonal entries have modulus one. Thus \(D\) is unitary.
\end{proof}
\begin{remark}\label{n3_remark}
The assumption \(n\geq 3\) in Lemma \ref{keylemma}, and thus in Theorems \ref{biholclassification} and \ref{mthm2}, is essential for us.
Indeed, when \(n=2\), the conclusion of that lemma is false.
Let
\(
\mathcal C=\{(z_1,z_2)\in \mathbb C^2: |z_1|=|z_2|\}.\)
Consider the matrix
\(
A=\begin{pmatrix}
\frac54 & \frac34\\[2mm]
\frac34 & \frac54
\end{pmatrix}.
\)
Then \(A\in GL(2,\mathbb C)\), but not a monomial matrix. 
We claim that \(A(\mathcal C)=\mathcal C\).  Write \(w=Az\), so that
\[
w_1=\frac54 z_1+\frac34 z_2,\qquad
w_2=\frac34 z_1+\frac54 z_2 .
\]
A direct computation gives
\[
\begin{aligned}
|w_1|^2-|w_2|^2
&=
\left|\frac54 z_1+\frac34 z_2\right|^2
-
\left|\frac34 z_1+\frac54 z_2\right|^2  \\
&=
\left(\frac{25}{16}-\frac{9}{16}\right)|z_1|^2
-
\left(\frac{25}{16}-\frac{9}{16}\right)|z_2|^2  \\
&=
|z_1|^2-|z_2|^2 .
\end{aligned}
\]
Thus \(A\) preserves the Hermitian quadratic form
\(Q(z)=|z_1|^2-|z_2|^2.\)
Since
\(\mathcal C=\{z\in\mathbb C^2: Q(z)=0\},\)
it follows that \(A(\mathcal C)\subset \mathcal C\).  Because \(A\) is invertible and the same
identity \(Q(Az)=Q(z)\) implies \(Q(A^{-1}z)=Q(z)\), we also have
\(\mathcal C\subset A(\mathcal C)\).  Hence \(A(\mathcal C)=\mathcal C\).
\end{remark}
\begin{lemma}\label{samebergmansp}
Let \(n\geq2\), and let
\(a_1,a_2>\frac{2}{n-1}.\)
For \(j=1,2\), let \(\Omega_j\subset\mathbb C^n\) be an
\(a_j\)-sandwichable domain. If there exists a biholomorphism
\(F:\Omega_1\longrightarrow\Omega_2,\)
then there exists an integer \(q\geq1\) such that the two domains have the same polynomial Bergman model
\(\mathcal P_{\leq q}\).
\end{lemma}

\begin{proof}
Since \(\Omega_j\) is \(a_j\)-sandwichable, there exist positive
numbers \(t_j^-\leq t_j^+\) such that
\[
D_{a_j,t_j^-}
\subset
\Omega_j
\subset
D_{a_j,t_j^+},
\qquad j=1,2.
\]

For each \(j=1,2\), let \(q_j\geq1\) be the unique integer satisfying
\[
\frac{q_j+1}{n-1}<
a_j
\leq
\frac{q_j+2}{n-1}.
\]
By Corollary~\ref{BergmanspaceofDat} and
Theorem~\ref{sandwich}, we therefore have
\[
A^2(\Omega_j)
=\mathcal P_{\leq q_j},
\qquad j=1,2.
\]
In particular,
\[
\dim A^2(\Omega_j)
=
\dim\mathcal P_{\leq q_j}
=
\binom{n+q_j}{q_j}.
\]

Let \(J_F(z)=\det F'(z),\) then the biholomorphic change-of-variables formula shows that
\[
U_F:A^2(\Omega_2)\longrightarrow A^2(\Omega_1),
\qquad
U_F f(z):=J_F(z)f(F(z)),
\]
is unitary. Hence
\(\dim A^2(\Omega_1)=\dim A^2(\Omega_2).\)
Consequently,
\[
\binom{n+q_1}{q_1}=
\binom{n+q_2}{q_2}.
\]
Since the function
\(q\longmapsto\binom{n+q}{q}\)
is strictly increasing on \(\mathbb Z_{\geq0}\), it follows that
\(q_1=q_2.\)
Writing their common value as \(q\), we obtain
\(
A^2(\Omega_j)=
\mathcal P_{\leq q},
~ j=1,2,\)
as claimed.
\end{proof}
\begin{proof}[Proof of Theorem~\ref{biholclassification}]
By Lemma~\ref{samebergmansp}, the two sandwichable domains have the same
Bergman space \(\mathcal P_{\leq q}\). Proposition~\ref{biholrigidity} therefore shows
that every biholomorphism \(F:\Omega_1\to\Omega_2\) is linear. By the horizon
cone argument preceding Lemma~\ref{keylemma}, its matrix preserves
\(\mathcal C\). Lemma~\ref{keylemma} now gives
\(
F(z)=\rho DPz,
\)
where \(\rho>0\), \(D\) is diagonal unitary, and \(P\) is a permutation
matrix.

Next, we show \(a_1=a_2\).

We first show that \(a_2\leq a_1\). Since
\(D_{a_1,t_1^-}\subset\Omega_1\text{ and }\Omega_2\subset D_{a_2,t_2^+},\)
we have
\(F(D_{a_1,t_1^-})\subset D_{a_2,t_2^+}\).

Write
\(
F(z)=\rho DPz,~
(Pz)_k=z_{\sigma(k)}.
\)
Since \(n\geq3\), there exists \(k<n\) such that
\(\sigma(k)\neq n\). For \(r>0\), set
\[
\delta_r
:=
\frac{t_1^-}{2(1+r^2)^{a_1}},
\]
and choose \(z^{(r)}\in\mathbb C^n\) such that
\[
|z_{\sigma(k)}^{(r)}|^2=r^2+\delta_r,
\qquad
|z_j^{(r)}|^2=r^2
\quad
(j\neq\sigma(k)).
\]
Because \(\sigma(k)\neq n\), we have
\(z^{(r)}\in D_{a_1,t_1^-}.\)
Hence
\(F(z^{(r)})\in D_{a_2,t_2^+}.\)
Since \(\sigma(n)\neq\sigma(k)\), this gives
\[
\frac{\rho^2t_1^-}{2(1+r^2)^{a_1}}
=
\rho^2\delta_r
<
\frac{t_2^+}{(1+\rho^2r^2)^{a_2}}.
\]
Therefore
\[
\frac{\rho^2t_1^-}{2t_2^+}
\frac{(1+\rho^2r^2)^{a_2}}
     {(1+r^2)^{a_1}}<1
\qquad(r>0).
\]
If \(a_2>a_1\), the left-hand side tends to \(+\infty\) as
\(r\to\infty\), a contradiction. Thus
\(a_2\leq a_1.\) 
Applying the same argument to
\(
F^{-1}:\Omega_2\longrightarrow\Omega_1
\)
gives \(a_1\leq a_2\). Consequently,
\(a_1=a_2.\)

Finally, assume in addition that \(\Omega_1=\Omega_2=: \Omega\). Since
\(1\in A^2(\Omega)\), the domain \(\Omega\) has finite positive volume. The
identity \(F(\Omega)=\Omega\) gives
\(
\operatorname{Vol}(\Omega)
=|\det F|^2\operatorname{Vol}(\Omega).\)
Hence \(|\det F|=1\). Since \(|\det(\rho DP)|=\rho^n\), we obtain
\(\rho=1\). 
\end{proof}

\section{General perturbation construction}\label{sec:general-perturbations}

Fix \(n\ge3\), \(a>2/(n-1)\), and positive numbers \(0<t_1<t_2\). We now
construct the perturbation family used throughout the remainder of the paper.
The construction is stated directly in its general form; the one-parameter
family used to prove Theorem~\ref{mthm2} will be obtained as a special case in
Section~\ref{sec:applications}.

For \(k=1,\ldots,n-1\), define
\begin{equation}\label{function_uk}
u_k(z):=(1+|z_n|^2)^{2a}
\bigl(|z_k|^2-|z_n|^2\bigr)^2.
\end{equation}
Then
\[
D_{a,t}=\{z\in\mathbb C^n:u_k(z)<t^2\text{ for all }k<n\}.
\]
Choose a smooth non-increasing cutoff function
\(
\chi\in C^\infty([0,\infty),[0,1])
\)
such that
\(
\chi=1 \text{ on }[0,t_1^2],\text{ and }
\chi=0 \text{ on }[t_2^2,\infty),
\)
and
\(
\chi'(s)<0 \text{ for } t_1^2<s<t_2^2.
\)
Define
\begin{equation}\label{function_R}
R(z):=\prod_{k=1}^{n-1}\chi(u_k(z)).
\end{equation}
Then \(R\in C^\infty(\mathbb C^n,[0,1])\), and \(R\) is invariant under
coordinatewise rotations and permutations of the first \(n-1\) coordinates.
Moreover,
\begin{equation}\label{eq:R-inner-outer}
R=1\quad\text{on }D_{a,t_1},
\qquad
R=0\quad\text{on }\mathbb C^n\setminus D_{a,t_2}.
\end{equation}

Define
\begin{equation}\label{function_alpha}
\alpha(s)=
\begin{cases}
\exp\!\left(-\dfrac{1}{s(1-s)}\right),&0<s<1,\\[6pt]
0,&s=0,1.
\end{cases}
\end{equation}
Then \(\alpha\in C^\infty([0,1])\),
\[
\alpha(0)=\alpha(1)=0,
\qquad
\alpha(s)>0\quad(0<s<1),
\]
and
\begin{equation}\label{eq:alpha-bound}
0\le \alpha(s)\le s(1-s)
\qquad(0\le s\le1).
\end{equation}
Indeed, putting \(x=s(1-s)\in[0,1/4]\), the function \(e^{-1/x}/x\) is
increasing on \((0,1/4]\), and hence
\[
\frac{\alpha(s)}{s(1-s)}
=\frac{e^{-1/x}}{x}
\le 4e^{-4}<1.
\]

Let \(I\subset\{1,\ldots,n\}\) be nonempty. For each \(j\in I\), choose
\(
m_j\in\mathbb Z_{\ge1}, 
\varepsilon_j>0,
\)
and assume
\begin{equation}\label{eq:Lambda}
\Lambda:=\sum_{j\in I}\varepsilon_j<\frac12.
\end{equation}
For \(m\ge1\), set
\begin{equation}\label{eq:h-m}
h_m(w):=\operatorname{Re}(w^m)(1+|w|^2)^{-m/2},
\qquad
\Psi(z):=\sum_{j\in I}\varepsilon_j h_{m_j}(z_j).
\end{equation}
Define
\begin{equation}\label{eq:general-perturbation}
R_{I,\mathbf m,\boldsymbol\varepsilon}(z)
:=R(z)+\alpha(R(z))\Psi(z).
\end{equation}
When convenient for brevity, we write
\[
R_{\boldsymbol\varepsilon}:=R_{I,\mathbf m,\boldsymbol\varepsilon}.
\]
Since \(|h_m|\le1\), we have \(|\Psi|\le\Lambda\). Put
\[
L_\Lambda(s):=s-\Lambda s(1-s),
\qquad 0\le s\le1.
\]
The function \(L_\Lambda\) is strictly increasing because \(\Lambda<1/2\),
and \eqref{eq:alpha-bound} gives
\begin{equation}\label{eq:perturbation-estimate}
L_\Lambda(R(z))
\le R_{\boldsymbol\varepsilon}(z)
\le R(z)+\Lambda R(z)(1-R(z)).
\end{equation}
In particular, the perturbation \(\alpha(R(z))\Psi(z)\) vanishes wherever \(R\in\{0,1\}\). Thus
\begin{equation}\label{eq:perturbation-inner-outer}
R_{\boldsymbol\varepsilon}=1\quad\text{on }D_{a,t_1},
\qquad
R_{\boldsymbol\varepsilon}=0
\quad\text{on }\mathbb C^n\setminus D_{a,t_2}.
\end{equation}

Let \(c\in(0,1)\) be a regular value of
\(R_{I,\mathbf m,\boldsymbol\varepsilon}\), and let
\(\Omega_{I,\mathbf m,\boldsymbol\varepsilon,c}\) denote the connected
component of
\[
\{R_{I,\mathbf m,\boldsymbol\varepsilon}>c\}
\]
that contains \(D_{a,t_1}\). Such regular values form a full-measure subset of
\((0,1)\) by Sard's theorem.

We next record, once and for all, the coordinate-axis information used both
to detect the loss of Reinhardt symmetry and to compute the automorphism
group. Put
\[
J:=\{1,\ldots,n-1\},
\qquad
I_0:=I\cap J,
\]
and, for \(q=1,\ldots,n\) and \(r>0\), define
\begin{equation}\label{eq:Theta-definition}
\Theta_q(r):=
\{\theta\in\mathbb R/2\pi\mathbb Z:
re^{i\theta}\mathbf e_q\in
\Omega_{I,\mathbf m,\boldsymbol\varepsilon,c}\}.
\end{equation}
For \(m\ge1\), let
\[
 S_m:=\{\theta\in\mathbb R/2\pi\mathbb Z:\cos(m\theta)>0\}.
\]
Since \(\chi\) is strictly decreasing on \((t_1^2,t_2^2)\), there is a
unique \(s_0\) in this interval such that \(\chi(s_0)=c\). Set
\begin{equation}\label{eq:axis-notation}
r_0:=s_0^{1/4},
\qquad
B(r):=\chi(r^4),
\qquad
\kappa_m(r):=r^m(1+r^2)^{-m/2}.
\end{equation}
For \(r\in(\sqrt{t_1},\sqrt{t_2})\), one has \(0<B(r)<1\).
Thus, for every \(j\in I_0\), define
\begin{equation}\label{eq:beta-definition}
\beta_j:(\sqrt{t_1},\sqrt{t_2})\longrightarrow\mathbb R,
\qquad
\beta_j(r):=
\frac{c-B(r)}
{\alpha(B(r))\varepsilon_j\kappa_{m_j}(r)}.
\end{equation}
In particular, \(\beta_j(r_0)=0\). Also set
\begin{equation}\label{eq:distinguished-profile-notation}
Q(r):=\chi\bigl((1+r^2)^{2a}r^4\bigr)^{n-1}\quad(r\ge0),
\qquad
\rho_n:=\inf\{r>0:Q(r)\le c\}.
\end{equation}
Then we have
\begin{lemma}\label{lem:axis-profiles}
For every \(j\in J\),
\begin{equation}\label{eq:axis-profile-r0}
\Theta_j(r_0)=
\begin{cases}
 S_{m_j},&j\in I_0,\\
\varnothing,&j\in J\setminus I_0.
\end{cases}
\end{equation}
Moreover, there exists \(\eta\in(0,\sqrt{t_2}-r_0)\) such that, for every
\(j\in I_0\), the function \(\beta_j\) is strictly increasing on
\([r_0,r_0+\eta)\), satisfies \(0<\beta_j(r)<1\) for
\(r_0<r<r_0+\eta\), and
\begin{equation}\label{eq:axis-profile-beta}
\Theta_j(r)=
\{\theta\in\mathbb R/2\pi\mathbb Z:
\cos(m_j\theta)>\beta_j(r)\}
\qquad(r_0<r<r_0+\eta).
\end{equation}
If \(n\in I\), then
\begin{equation}\label{eq:axis-profile-distinguished}
\Theta_n(\rho_n)= S_{m_n}.
\end{equation}
\end{lemma}

\begin{proof}
On an ordinary coordinate axis we have
\begin{equation}\label{eq:ordinary-axis-formula}
R_{\boldsymbol\varepsilon}(re^{i\theta}\mathbf e_j)
=
\begin{cases}
B(r),&j\in J\setminus I_0,\\[2mm]
B(r)+\alpha(B(r))\varepsilon_j\kappa_{m_j}(r)
\cos(m_j\theta),&j\in I_0.
\end{cases}
\end{equation}
At \(r=r_0\), the unperturbed expression equals \(c\). If \(j\in I_0\)
and \(\cos(m_j\theta)>0\), the radial segment from the origin to
\(r_0e^{i\theta}\mathbf e_j\) lies in the superlevel set: along the segment
\(B\ge c\), and the perturbation is nonnegative and is strictly positive
where \(B=c\). Hence the endpoint lies in the selected component. The reverse
inclusion follows directly from the strict superlevel inequality. This proves
\eqref{eq:axis-profile-r0}.

Moreover,
\[
B'(r_0)=4r_0^3\chi'(s_0)<0,
\qquad
\beta_j'(r_0)=
\frac{-B'(r_0)}
{\alpha(c)\varepsilon_j\kappa_{m_j}(r_0)}>0.
\]
Since \(I_0\) is finite, a common
\(\eta\in(0,\sqrt{t_2}-r_0)\) can be chosen so that all \(\beta_j\) have the
asserted monotonicity and range. Formula
\eqref{eq:ordinary-axis-formula} shows that membership in the selected
component implies the inequality in \eqref{eq:axis-profile-beta}. Conversely,
if that inequality holds, then for \(r_0\le t\le r\),
\[
\beta_j(t)\le\beta_j(r)<\cos(m_j\theta),
\]
so the segment from radius \(r_0\) to radius \(r\) remains in the superlevel
set. The first part of the proof connects its initial point to the origin.
This proves \eqref{eq:axis-profile-beta}.

Finally, assume \(n\in I\). The function \(Q\) satisfies
\[
Q(\rho_n)=c,
\qquad
Q(r)>c\quad(0\le r<\rho_n).
\]
Along the distinguished axis,
\[
R_{\boldsymbol\varepsilon}(\rho_ne^{i\theta}\mathbf e_n)
=c+\alpha(c)\varepsilon_n
\kappa_{m_n}(\rho_n)\cos(m_n\theta).
\]
For \(\cos(m_n\theta)>0\), the radial segment to this endpoint lies in the
superlevel set, while the reverse inclusion again follows from the strict
inequality. Thus \eqref{eq:axis-profile-distinguished} holds.
\end{proof}

\begin{proposition}\label{prop:general-construction}
For every regular value \(c\in(0,1)\), the set
\(\Omega_{I,\mathbf m,\boldsymbol\varepsilon,c}\) is a smooth domain satisfying
\begin{equation}\label{eq:sandwich-inclusion}
D_{a,t_1}
\subset
\Omega_{I,\mathbf m,\boldsymbol\varepsilon,c}
\subset
D_{a,t_2}.
\end{equation}
Moreover, it is not Reinhardt. If, in addition,
\(
\frac{2}{n-1}<a\le\frac{3}{n-1},\)
then
\(
A^2(\Omega_{I,\mathbf m,\boldsymbol\varepsilon,c})
=\operatorname{Span}\{1,z_1,\ldots,z_n\},\)
and its Bergman metric has constant positive holomorphic sectional curvature
\(2\).
\end{proposition}

\begin{proof}
The inclusions in \eqref{eq:sandwich-inclusion} follow immediately from
\eqref{eq:perturbation-inner-outer}. Since \(c\) is a regular value, the level set
\[
\{R_{I,\mathbf m,\boldsymbol\varepsilon}=c\}
\]
is a smooth, i.e., $C^\infty$ and non-singular, real hypersurface. The boundary of the selected
connected component is contained in this level set. At every boundary point,
the implicit function theorem identifies the superlevel set locally with one
side of a smooth hypersurface; that local side is connected and therefore lies
in the selected global component. Hence the boundary is locally the entire
level hypersurface, which is smooth.

To see that the domain is not Reinhardt, choose \(j\in I\). If \(j<n\),
\eqref{eq:axis-profile-r0} shows that the angular trace of its intersection
with the circle \(\{r_0e^{i\theta}\mathbf e_j\}\) is the nonempty proper set
\( S_{m_j}\).
If \(j=n\), the same conclusion follows from
\eqref{eq:axis-profile-distinguished}. Either conclusion is incompatible with
coordinatewise rotational invariance.

Finally, when \(
\frac{2}{n-1}<a\le\frac{3}{n-1}\), Theorem~\ref{sandwich} and
Corollary~\ref{isom_cor} give the stated Bergman space and curvature.
\end{proof}

\section{Automorphism groups of perturbation domains}\label{sec:perturbation-aut}
Retain the notation of Section~\ref{sec:general-perturbations}, and write
\[
\Omega:=\Omega_{I,\mathbf m,\boldsymbol\varepsilon,c},
\qquad
U:=J\setminus I_0.
\]
We continue to call the coordinates indexed by \(J\) ordinary and \(z_n\)
distinguished.
Thus \(I_0\) is the set of perturbed ordinary coordinates, and \(U\) is the set
of unperturbed ordinary coordinates.

Let
\(I_0\longrightarrow \mathbb Z_{\geq 1}\times \mathbb R_{>0},~
j\longmapsto (m_j,\varepsilon_j)\)
be the assignment of perturbation data.  Denote its distinct values by
\(
(d_1,\eta_1),\ldots,(d_s,\eta_s).\)
For each \(\nu=1,\ldots,s\), set
\[
E_\nu:=\{j\in I_0:(m_j,\varepsilon_j)=(d_\nu,\eta_\nu)\},
\qquad
r_\nu:=|E_\nu|.
\]
Thus \(I_0=\bigsqcup_{\nu=1}^s E_\nu\), and \(r_\nu\) is the multiplicity with
which the perturbation data \((d_\nu,\eta_\nu)\) occur.

Set
\[
K_n:=
\begin{cases}
\mathbb T, & n\notin I,\\
\mu_{m_n}, & n\in I.
\end{cases}
\]
Here
\[
\mathbb T=\{\lambda\in\mathbb C:|\lambda|=1\},
\qquad
\mu_m=\{\zeta\in\mathbb C:\zeta^m=1\}.
\]
We will denote finite permutation groups with the symbol $\Sigma$.

\begin{theorem}\label{thm:perturbed-aut-group}
Under the assumptions above,
\begin{equation}\label{eq_wholeautgp}
\operatorname{Aut}(\Omega)
\cong
\left(
\prod_{\nu=1}^s
(\mu_{d_\nu}^{r_\nu}\rtimes \Sigma_{r_\nu})
\right)
\times
(\mathbb T^{|U|}\rtimes \Sigma_{|U|})
\times
K_n,
\end{equation}
with the convention that empty products and factors of exponent \(0\) are
trivial.
\end{theorem}
\begin{remark}
The semidirect product groups \(\mu_{d_\nu}^{r_\nu}\rtimes \Sigma_{r_\nu}\) and \(\mathbb T^{|U|}\rtimes \Sigma_{|U|} \) are {\it wreath products} as defined in \cite[Section 4.1]{JamesKerber}. They are often denoted by \(\mu_{d_\nu} \wr \Sigma_{r_\nu}\) and \(\mathbb T\wr \Sigma_{|U|} \), respectively. It appears to be an interesting byproduct of this investigation that it produces examples of unbounded domains whose Bergman metric has constant positive holomorphic sectional curvature and whose (finite) automorphism groups are products of such wreath products.
\end{remark}

To prove this theorem, we need the following lemmas.
\begin{lemma}\label{lem:circle-stabilizer}
For every \(m\ge1\), the angular set \( S_m\) defined in
Section~\ref{sec:general-perturbations} satisfies
\(
\{e^{it}\in\mathbb T: S_m+t= S_m\}=\mu_m.
\)
\end{lemma}

\begin{proof}
Suppose \( S_m+t= S_m\). Then translation by \(t\) preserves
the boundary
\[
\partial S_m=
\left\{
\frac{\pi/2+k\pi}{m}:k=0,\ldots,2m-1
\right\}.
\]
Hence
\(
mt\in\pi\mathbb Z.
\)
If \(mt\) were an odd multiple of \(\pi\), then
\(\cos(m(\theta+t))=-\cos(m\theta),\)
so \( S_m+t\) would be the complementary open set, not
\( S_m\). Therefore
\(
mt\in 2\pi\mathbb Z,\)
which is equivalent to \(e^{it}\in\mu_m\). The converse inclusion is immediate.
\end{proof}
Next we show the rigidity of the distinguished coordinate under the automorphisms.
\begin{lemma}
\label{lem:linear-distinguished}
Let
\(F\in\operatorname{Aut}(\Omega)\) 
be an automorphism of \(\Omega\) of the form
\(
F(z)=Az,
~
A=DP,
\)
where
\(
D=\operatorname{diag}(\lambda_1,\ldots,\lambda_n),
~
\lambda_j\in\mathbb T,
\)
and \(P\) is a permutation matrix. If \(\sigma\in \Sigma_n\) is defined by
\(
(Pz)_k=z_{\sigma(k)},
\)
then
\(
\sigma(n)=n.
\)
\end{lemma}

\begin{proof}
Assume, toward a contradiction, that
\(
\sigma(n)=j<n.
\)

Put
\[
y:=c^{1/(n-1)}.
\]
Since \(0<y<1\) and \(n\ge3\),
\[
L_\Lambda(y)
>
\frac{y(1+y)}2
>
y^2
\ge
y^{n-1}
=
c.
\]
Choose \(x\in(0,y)\), sufficiently close to \(y\), such that
\(L_\Lambda(x)>c\). Since \(x<y\), we also have \(x^{n-1}<c\); thus
\begin{equation}\label{eq:x-choice}
L_\Lambda(x)>c,
\qquad
x^{n-1}<c.
\end{equation}
Choose \(\tau\in(t_1^2,t_2^2)\) with
\[
\chi(\tau)=x,
\]
and, for \(r>0\), put
\[
\delta_r:=\sqrt{\tau}(1+r^2)^{-a}.
\]

For \(0\le s\le\delta_r\), choose \(p_s^{(r)}\in\mathbb C^n\) whose phases
are independent of \(s\) and whose squared moduli satisfy
\[
\left|(p_s^{(r)})_k\right|^2
=
\begin{cases}
r^2+s, & k=j,\\
r^2, & k\ne j.
\end{cases}
\]
Set
\[
p^{(r)}:=p_{\delta_r}^{(r)}.
\]
Then \(p_0^{(r)}\in D_{a,t_1}\). For \(0\le s\le\delta_r\), the only
possibly nontrivial ordinary-coordinate factor in \(R(p_s^{(r)})\) is the
\(j\)-th one, whose argument of \(\chi\) equals
\[
(1+r^2)^{2a}s^2\le(1+r^2)^{2a}\delta_r^2=\tau.
\]
Since \(\chi\) is non-increasing,
\[
R(p_s^{(r)})\ge\chi(\tau)=x.
\]
By \eqref{eq:perturbation-estimate} and the monotonicity of \(L_\Lambda\),
\[
R_{\boldsymbol\varepsilon}(p_s^{(r)})
\ge
L_\Lambda(R(p_s^{(r)}))
\ge
L_\Lambda(x)
>
c.
\]
Thus \(s\mapsto p_s^{(r)}\) lies in \(\{R_{\boldsymbol\varepsilon}>c\}\) and joins
\(p^{(r)}\) to \(D_{a,t_1}\). Hence
\[
p^{(r)}\in\Omega.
\]

This membership argument is valid for every fixed choice of phases. After
\(A=DP\) has been fixed, choose the common phases of the family
\(p_s^{(r)}\). Since the indices
\(\sigma(\ell)\), \(\ell\in I\), are pairwise distinct, these phases may be
chosen independently so that
\[
\operatorname{Re}\left((Ap^{(r)})_\ell^{m_\ell}\right)=0
\qquad(\ell\in I).
\]
The perturbation term in \(R_{\boldsymbol\varepsilon}(Ap^{(r)})\) therefore vanishes.

Because \(\sigma(n)=j\) and \(\sigma\) is injective,
\[
|(Ap^{(r)})_n|^2=r^2+\delta_r,
\qquad
|(Ap^{(r)})_k|^2=r^2
\quad(k<n).
\]
Consequently, every ordinary-coordinate factor in \(R(Ap^{(r)})\) has the
same argument
\[
\tau_r
:=
(1+r^2+\delta_r)^{2a}\delta_r^2
=
\tau
\left(
\frac{1+r^2+\delta_r}{1+r^2}
\right)^{2a}.
\]
Since \(\delta_r/(1+r^2)\to0\), we have \(\tau_r\to\tau\), and hence
\[
R_{\boldsymbol\varepsilon}(Ap^{(r)})
=
R(Ap^{(r)})
=
\chi(\tau_r)^{n-1}
\longrightarrow
\chi(\tau)^{n-1}
=
x^{n-1}
<
c.
\]
Thus \(Ap^{(r)}\notin\Omega\) for all sufficiently large \(r\), contradicting
\(A(\Omega)=\Omega\). Therefore \(\sigma(n)=n\).
\end{proof}

\begin{lemma}
\label{lem:monomial-stabilizer}
Let
\(
A=DP
\)
be a monomial map, where \(D=\operatorname{diag}(\lambda_1,\ldots,\lambda_n),
~
\lambda_j\in\mathbb T,\) and define \(\sigma\in \Sigma_n\) by
\(
(Pz)_k=z_{\sigma(k)}.
\)
Assume \(\sigma(n)=n\). Then \(A\in\operatorname{Aut}(\Omega)\) if and only if
\[
\sigma(U)=U,
\qquad
\sigma(E_\nu)=E_\nu
\quad(\nu=1,\ldots,s),
\]
\[
\lambda_j\in\mu_{d_\nu}
\quad(j\in E_\nu),
\qquad
\lambda_n\in K_n.
\]
There is no further restriction on \(\lambda_j\) for \(j\in U\).
\end{lemma}

\begin{proof}
Assume first that \(A\in\operatorname{Aut}(\Omega)\), and write
\(\lambda_q=e^{i\phi_q}\). Since
\[
A\bigl(re^{i\theta}\mathbf e_{\sigma(q)}\bigr)
=
re^{i(\theta+\phi_q)}\mathbf e_q,
\]
the identity \(A(\Omega)=\Omega\) gives
\begin{equation}\label{eq:axis-transport}
\Theta_{\sigma(q)}(r)+\phi_q=\Theta_q(r)
\qquad(q=1,\ldots,n,\ r>0).
\end{equation}

If \(q\in U\), then \eqref{eq:axis-transport} and
\eqref{eq:axis-profile-r0} imply
\(\Theta_{\sigma(q)}(r_0)=\varnothing\). Since \(\sigma(n)=n\), we have
\(\sigma(q)\in J\), and another application of
\eqref{eq:axis-profile-r0} gives \(\sigma(q)\in U\). Thus
\[
\sigma(U)=U,
\qquad
\sigma(I_0)=I_0.
\]
If \(q\in I_0\), then \eqref{eq:axis-transport} and
\eqref{eq:axis-profile-r0} show that \( S_{m_{\sigma(q)}}\) is a
rotation of \( S_{m_q}\). Since \( S_m\) has exactly \(m\)
connected components,
\begin{equation}\label{eq:m-preserved}
m_{\sigma(q)}=m_q.
\end{equation}

It remains to recover the coefficients \(\varepsilon_j\). Fix \(q\in I_0\),
put \(m:=m_q=m_{\sigma(q)}\), and take \(r\in(r_0,r_0+\eta)\), with
\(\eta\) as in Lemma~\ref{lem:axis-profiles}. The sets
\(\Theta_q(r)\) and \(\Theta_{\sigma(q)}(r)\) differ by a rotation, so their
component lengths are equal. By \eqref{eq:axis-profile-beta},
\[
\frac{2}{m}\arccos(\beta_q(r))
=
\frac{2}{m}\arccos(\beta_{\sigma(q)}(r)).
\]
Thus \(\beta_q(r)=\beta_{\sigma(q)}(r)\). All factors in the definition of
\(\beta_j(r)\), except \(\varepsilon_j\), depend only on \(m_j\); hence
\[
\varepsilon_{\sigma(q)}=\varepsilon_q.
\]
Together with \eqref{eq:m-preserved}, this shows that \(\sigma\) preserves
every block \(E_\nu\).

We now determine the diagonal factors. Let \(j\in E_\nu\). Since
\(\sigma(j)\in E_\nu\), equations \eqref{eq:axis-transport} and
\eqref{eq:axis-profile-r0} give
\[
 S_{d_\nu}+\phi_j= S_{d_\nu}.
\]
Lemma~\ref{lem:circle-stabilizer} yields
\[
\lambda_j\in\mu_{d_\nu}.
\]
For \(j\in U\), there is no additional restriction beyond
\(\lambda_j\in\mathbb T\).

Since \(\sigma(n)=n\), equation \eqref{eq:axis-transport} gives
\[
\Theta_n(r)+\phi_n=\Theta_n(r)
\qquad(r>0).
\]
If \(n\notin I\), no perturbation term involves \(z_n\), so there is no
further restriction on \(\lambda_n\), and \(\lambda_n\in\mathbb T=K_n\).
Assume now that \(n\in I\). Applying \eqref{eq:axis-transport} at
\(r=\rho_n\), using \eqref{eq:axis-profile-distinguished}, and then applying
Lemma~\ref{lem:circle-stabilizer}, we obtain
\[
\lambda_n\in\mu_{m_n}=K_n.
\]
This proves the necessary conditions.

Conversely, suppose that the stated permutation and diagonal conditions hold.
Recall the functions \(h_m\) and \(\Psi\) from \eqref{eq:h-m}; by
\eqref{eq:general-perturbation},
\(R_{\boldsymbol\varepsilon}=R+\alpha(R)\Psi\).
The base function \(R\) is invariant under coordinatewise rotations and under
permutations of the ordinary coordinates. Since \(\sigma(n)=n\),
\[
R(Az)=R(z).
\]
For each block \(E_\nu\),
\[
\begin{aligned}
\sum_{j\in E_\nu}\varepsilon_j h_{m_j}((Az)_j)
&=
\eta_\nu\sum_{j\in E_\nu}
h_{d_\nu}(\lambda_j z_{\sigma(j)})\\
&=
\eta_\nu\sum_{j\in E_\nu}
h_{d_\nu}(z_{\sigma(j)})\\
&=
\eta_\nu\sum_{j\in E_\nu}h_{d_\nu}(z_j).
\end{aligned}
\]
Here the second equality uses \(\lambda_j^{d_\nu}=1\), and the third uses
that \(\sigma\) permutes \(E_\nu\). There are no perturbation terms indexed
by \(U\). If \(n\in I\), then \(\sigma(n)=n\) and
\(\lambda_n^{m_n}=1\), so the distinguished term is also invariant. Hence
\[
\Psi(Az)=\Psi(z),
\qquad
R_{\boldsymbol\varepsilon}(Az)=R_{\boldsymbol\varepsilon}(z).
\]
Thus \(A\) preserves \(\{R_{\boldsymbol\varepsilon}>c\}\). It also preserves
\(D_{a,t_1}\), and therefore preserves the connected component containing
\(D_{a,t_1}\), namely \(\Omega\). Hence
\(
A\in\operatorname{Aut}(\Omega).
\)
\end{proof}

\begin{proof}[Proof of Theorem~\ref{thm:perturbed-aut-group}]
By \eqref{eq:sandwich-inclusion} and Theorem~\ref{biholclassification}, every
\(F\in\operatorname{Aut}(\Omega)\) is of the form
\[
F(z)=DPz,
\]
where \(D=\operatorname{diag}(\lambda_1,\ldots,\lambda_n)\) is unitary
diagonal and \(P\) is a permutation matrix. By
Lemma~\ref{lem:linear-distinguished}, the associated permutation satisfies
\(\sigma(n)=n\). Lemma~\ref{lem:monomial-stabilizer} then shows that 
\[
\sigma\in
\left(\prod_{\nu=1}^s\operatorname{Sym}(E_\nu)\right)
\times\operatorname{Sym}(U)
\]
and
\[
(\lambda_j)_{j\in E_\nu}\in\mu_{d_\nu}^{r_\nu}
\quad(\nu=1,\ldots,s),
\qquad
(\lambda_j)_{j\in U}\in\mathbb T^{|U|},
\qquad
\lambda_n\in K_n.
\]
Each \(\operatorname{Sym}(E_\nu)\cong \Sigma_{r_\nu}\) acts by permuting the
factors of \(\mu_{d_\nu}^{r_\nu}\), and
\(\operatorname{Sym}(U)\cong \Sigma_{|U|}\) acts similarly on
\(\mathbb T^{|U|}\). Therefore
\[
\operatorname{Aut}(\Omega)
\cong
\left(
\prod_{\nu=1}^s
(\mu_{d_\nu}^{r_\nu}\rtimes \Sigma_{r_\nu})
\right)
\times
(\mathbb T^{|U|}\rtimes \Sigma_{|U|})
\times
K_n.
\]\end{proof}

\section{Non-Reinhardt Domains with Trivial Automorphism Group, Proof of Theorem \ref{mthm2}}\label{sec:applications}
We now isolate the one-parameter family that proves the abundance
claim in Theorem~\ref{mthm2}. \par
Take \(I=\{1,\ldots,n\}\), let \(m_1=\cdots=m_n=1\), and choose positive
numbers \(\varepsilon_1,\ldots,\varepsilon_n\) such that
\(
\sum_{j=1}^n\varepsilon_j<\frac12
\)
and \(\varepsilon_1,\ldots,\varepsilon_{n-1}\) are pairwise distinct.

 Write
\(
R_{\boldsymbol\varepsilon}:=
R_{I,\mathbf 1,\boldsymbol\varepsilon}.
\)
Let
\(
\mathcal R
:=
\left\{
c\in(0,1):
c\text{ is a regular value of }R_{\boldsymbol\varepsilon}
\right\}.
\)
By Sard's theorem, \(\mathcal R\) has full measure in \((0,1)\) and is uncountable.\par

For \(c\in\mathcal R\), set
\(\Omega_c:=\Omega_{I,\mathbf 1,\boldsymbol\varepsilon,c}.\) We will prove
Theorem~\ref{mthm2} by showing that the domains \(\Omega_c\) are pairwise
biholomorphically inequivalent for distinct regular values \(c\).
\begin{proof}[Proof of Theorem \ref{mthm2}]
Proposition~\ref{prop:general-construction} gives the asserted
smoothness, sandwich, non-Reinhardt, and curvature properties. Moreover,
Theorem~\ref{thm:perturbed-aut-group} gives
\[
\operatorname{Aut}(\Omega_c)=\{\operatorname{id}\},
\]
because every coordinate is perturbed with degree \(1\), i.e., \(m_1=\cdots=m_n=1\), the ordinary
perturbation amplitudes are pairwise distinct, and
\(\mu_1=\{1\}\).

It remains to prove that distinct regular values give biholomorphically
inequivalent domains. Suppose that
\(
F:\Omega_{c_1}\longrightarrow\Omega_{c_2}
\)
is a biholomorphism. By Theorem~\ref{biholclassification},
\(
F(z)=\rho DPz,
\)
where \(\rho>0\), \(D\) is diagonal unitary, and \(P\) is a permutation
matrix. Let \(\sigma\in \Sigma_n\) be determined by
\(
(Pz)_k=z_{\sigma(k)}.
\)

Put
\(
s_i:=\chi^{-1}(c_i),
~ i=1,2,
\)
where the inverse is taken on the interval
\((t_1^2,t_2^2)\). Choose
\[
q\in\{1,\ldots,n-1\}\setminus\{\sigma(n)\},
\]
which is possible because \(n\geq3\). For \(r>0\), set
\[
\delta_r:=\frac{\sqrt{s_1}}{(1+r^2)^a}
\]
and define \(z^{(r)}\in\mathbb C^n\) by
\[
(z^{(r)})_j=
\begin{cases}
i\sqrt{r^2+\delta_r},&j=q,\\[2mm]
ir,&j\neq q.
\end{cases}
\]

Since all the coordinates of \(z^{(r)}\) are purely imaginary and all
perturbation degrees are equal to \(1\), the perturbation term vanishes at both
\(z^{(r)}\) and \(-z^{(r)}\). Moreover,
\[
u_q(z^{(r)})=s_1,
\qquad
u_j(z^{(r)})=0\quad(j\neq q),
\]
and hence
\[
R_{\boldsymbol\varepsilon}(z^{(r)})
=R_{\boldsymbol\varepsilon}(-z^{(r)})
=\chi(s_1)
=c_1.
\]
Decreasing \(\delta_r\) to \(0\) joins these points to
\(D_{a,t_1}\) through the superlevel set. Thus
\(
z^{(r)},-z^{(r)}\in\partial\Omega_{c_1}.
\)

Since \(F\) is linear,
\(
F(-z^{(r)})=-F(z^{(r)}).\)
The unperturbed part \(R\) is even, whereas the perturbation term is odd
because all \(m_j=1\). Since
\(F(z^{(r)}),F(-z^{(r)})\in\partial\Omega_{c_2},\)
adding the two corresponding boundary equations gives
\(R(F(z^{(r)}))=c_2.\)

Let \(k_0=\sigma^{-1}(q)\). Our choice \(q\neq\sigma(n)\) implies
\(k_0<n\). Consequently, the distinguished coordinate of \(F(z^{(r)})\)
has squared modulus \(\rho^2r^2\), the \(k_0\)-th ordinary coordinate
has squared modulus \(\rho^2(r^2+\delta_r)\), and all other coordinates
have squared modulus \(\rho^2r^2\). Therefore
\[
c_2=R(F(z^{(r)}))
=
\chi\!\left(
s_1\rho^4
\left(
\frac{1+\rho^2r^2}{1+r^2}
\right)^{2a}
\right) \qquad (r>0).
\]
Since \(c_2\in(0,1)\) and \(\chi\) is strictly decreasing on its
transition interval, we obtain
\[
s_2=s_1\rho^4
\left(
\frac{1+\rho^2r^2}{1+r^2}
\right)^{2a}
\qquad(r>0).
\]

Letting \(r\to0\) gives
\(
s_2=s_1\rho^4,\)
whereas letting \(r\to\infty\) gives
\(s_2=s_1\rho^{4+4a}.\)
Since \(a>0\), \(s_1>0\), and \(\rho>0\), it follows that
\(\rho=1.\)
Hence
\(s_1=s_2,\)
and therefore
\[
c_1=\chi(s_1)=\chi(s_2)=c_2.
\]

Thus distinct parameters in \(\mathcal R\) give pairwise
biholomorphically inequivalent domains. Since
\(\mathcal R\) is uncountable, the conclusion follows.
\end{proof}
\section{The reference Reinhardt family}\label{sec:reference-family}

We conclude by recording the automorphism group and parameter rigidity of the
Reinhardt domains \(D_{a,t}\). In fact, Theorem \ref{biholclassification} yields the following corollary:
\begin{corollary}\label{autRein}
Let \(n\ge3\), \(a>2/(n-1)\), and \(t>0\). Then
\[
\operatorname{Aut}(D_{a,t})
\cong
\bigl(\mathbb T^{n-1}\rtimes \Sigma_{n-1}\bigr)\times\mathbb T.
\]
\end{corollary}

\begin{proof}
Theorem~\ref{biholclassification} shows that every automorphism is a monomial
unitary map \(A=DP\). Let \(\sigma\in \Sigma_n\) be determined by
\((Pz)_k=z_{\sigma(k)}\). We claim that \(\sigma(n)=n\).

Suppose instead that \(\sigma(n)=j<n\). Fix \(s\ge0\), and put
\[
T_s:=\frac{t}{(1+s)^a}.
\]
The function
\[
\delta\longmapsto \frac{t}{(1+s+\delta)^a}
\]
is continuous and strictly decreasing. Since its value at \(\delta=T_s\) is
strictly smaller than \(T_s\), we may choose \(\delta\in(0,T_s)\) such that
\begin{equation}\label{eq:reference-delta}
\delta>\frac{t}{(1+s+\delta)^a}.
\end{equation}
Choose \(z\in\mathbb C^n\) with
\[
|z_n|^2=s,
\qquad
|z_j|^2=s+\delta,
\qquad
|z_k|^2=s\quad(k<n,\ k\ne j).
\]
Since \(\delta<T_s\), we have \(z\in D_{a,t}\). On the other hand,
\[
|(Az)_n|^2=s+\delta,
\]
and, by injectivity of \(\sigma\), every ordinary coordinate of \(Az\) has
squared modulus \(s\). Thus membership of \(Az\) in \(D_{a,t}\) would require
\[
\delta<\frac{t}{(1+s+\delta)^a},
\]
contradicting \eqref{eq:reference-delta}. Hence \(\sigma(n)=n\).

It follows that the permutation is arbitrary on the first \(n-1\)
coordinates and fixes the last coordinate. Conversely, all such permutations,
together with arbitrary coordinatewise unitary rotations, preserve the
defining inequalities. This gives the stated semidirect product.
\end{proof}

\begin{theorem}\label{reference_domain_rigidity}
Let \(n\geq 3\) and let
\(
a_1,a_2>\frac{2}{n-1},~t_1,t_2>0.
\)
Then \(D_{a_1,t_1}\) and \(D_{a_2,t_2}\) are biholomorphic if and only if
\(a_1=a_2\text{ and }t_1=t_2.\)
\end{theorem}
\begin{proof}
The sufficient direction is immediate. Conversely, suppose that
\(F:D_{a_1,t_1}\longrightarrow D_{a_2,t_2}\)
is a biholomorphism. Since \(D_{a_1,t_1}\) and \(D_{a_2,t_2}\) are
\(a_1\)- and \(a_2\)-sandwichable, respectively,
Theorem~\ref{biholclassification} gives
\(a:=a_1=a_2\)
and
\(
F(z)=\rho DPz,\)
where \(\rho>0\), \(D\) is diagonal unitary, and \(P\) is a
permutation matrix.

For \(j=1,\ldots,n\), set
\(L_j:=\mathbb C\mathbf e_j,\)
and let \(s_{a,t_1}>0\) be the unique solution of
\(s(1+s)^a=t_1.\)
Then
\[
D_{a,t_1}\cap L_j
=
\{w\mathbf e_j:|w|^2<t_1\},
\qquad j<n,
\]
whereas
\[
D_{a,t_1}\cap L_n
=
\{w\mathbf e_n:|w|^2<s_{a,t_1}\}.
\]
Since
\(t_1=s_{a,t_1}(1+s_{a,t_1})^a>s_{a,t_1},\)
the axis \(L_n\) is the unique coordinate axis whose intersection
with \(D_{a,t_1}\) has smaller radius than the others. The analogous statements hold
for \(D_{a,t_2}\). Hence \(P\) fixes \(L_n\).

Comparison of the ordinary coordinate-axis sections gives
\(t_2=\rho^2t_1,\)
while comparison of the distinguished coordinate-axis sections gives
\(s_{a,t_2}=\rho^2s_{a,t_1}.\)
Put \(s:=s_{a,t_1}\). Then
\[
t_2=s_{a,t_2}(1+s_{a,t_2})^a
=
\rho^2s(1+\rho^2s)^a.
\]
On the other hand,
\[
t_2=\rho^2t_1=
\rho^2s(1+s)^a.
\]
Therefore
\[
(1+\rho^2s)^a=(1+s)^a.
\]
Since \(a>0\) and \(s>0\), it follows that
\(\rho=1\) and hence \(t_2=t_1.\)
\end{proof}

\end{document}